\documentclass[12pt]{amsart}
\usepackage{nccmath}

\usepackage{amsmath,amsthm,amssymb}
\usepackage{times}
\usepackage{graphicx}
\usepackage{enumerate}
\usepackage{enumitem}
\usepackage{bigints}
\usepackage{verbatim}
\usepackage{xcolor}

\newtheorem{theorem}{Theorem}[section]

\newtheorem{lemma}[theorem]{Lemma}
\newtheorem{proposition}[theorem]{Proposition}

\newcommand{\be}{\begin{equation}}
\newcommand{\ee}{\end{equation}}

\newcommand{\lt}{\left}
\newcommand{\rt}{\right}
\newcommand{\set}[1]{\lt\{ #1 \rt\}}
\newcommand{\pI}[1]{\lt< #1 \rt>}

\newcommand{\Sp}{\mathbb{S}}
\newcommand{\rr}{\mathbb{R}}
\newcommand{\nn}{\mathbb{N}}

\newcommand{\abs}[1]{\left|#1\right|}
\newcommand{\norm}[1]{\abs{\abs{#1}}}
\newcommand{\RNum}[1]{\uppercase\expandafter{\romannumeral #1\relax}}

\theoremstyle{plain}

\newtheorem{example}[theorem]{Examples}
\newtheorem{claim}[theorem]{Claim}
\newcommand{\restri}[2]{\left.#1\right|_{#2}}
\newcommand\restr[2]{{
		\left.\kern-\nulldelimiterspace 
		#1 
		\vphantom{\big|} 
		\right|_{#2} 
}}

\theoremstyle{definition}
\newtheorem{remark}[theorem]{Remark}

\date{\today}

  \usepackage{hyperref}
  \usepackage[alphabetic, msc-links, backrefs]{amsrefs}

\begin{document}
\setlength{\baselineskip}{1.2\baselineskip}

\title[Convexity Estimate and Examples]
{Convexity estimate for  translating solitons of  concave fully nonlinear extrinsic geometric flows in $\rr^{n+1}$}

\author{Jose Torres Santaella }
\address{Departmento de Matemáticas}
\email{jgtorre1@uc.cl}

\begin{abstract} 
The main result of this paper is a convexity estimate for translating solitons of extrinsic geometric flows which evolve under a $1$-homogeneous concave function in the principal curvatures.  In addition, we show examples of these hypersurfaces in $\rr^{n+1}$ for particular functions.
\end{abstract}

\maketitle

\section{Introduction}

Geometric evolution problems for hypersurfaces have had a remarkable development over the last years, since these kind of problems lead to interesting non-linear PDE's that have been used to solve important open questions in Mathematics and Physics.

In this paper we consider translating solutions to extrinsic geometric flows given by a concave symmetric function in their principal curvatures supported in a convex symmetric open cone. 
Being more precise, we say that a closed manifold $\Sigma^n$ evolves under the $\gamma$-flow in $\rr^{n+1}$ if, for a given immersion $F_0:\Sigma\to\rr^{n+1}$, there exist a $1$-parameter family of smooth immersions $F:\Sigma\times[0,T]\to\rr^{n+1}$ which solves
\begin{align}\label{gamma-flow}
\begin{cases}
\left(\partial_t F\right)^{\perp}&=\gamma(\lambda),\mbox{ on }\Sigma\times(0,T),
\\
F(\cdot,0)&=F_0(\cdot),
\end{cases}
\end{align}
where $(\cdot)^{\perp}$ denotes the orthogonal projection onto the normal bundle of $T\Sigma_t$ in $T\rr^{n+1}$, $\lambda=(\lambda_1,\ldots,\lambda)$ is the principal curvature vector, and $\nu$ is the outward unit normal vector of $\Sigma_t=F(\Sigma,t)$, respectively.
\newline
In addition, the function $\gamma:\Gamma\to\rr$ will satisfy the following properties:
\begin{enumerate}[label=(\alph*)]
	\item\label{1} $\Gamma\subset\rr^n$ is a symmetric open convex cone and $\gamma:\Gamma\to\rr$ is symmetric\footnote{Invariant under the $n$-permutation group in its variables.}, smooth and positive. 
	
	\item\label{2} $\gamma$ is strictly increasing in each variable, i.e: $\frac{\partial \gamma}{\partial\lambda_i}>0$ in $\Gamma$ for every $i$.  
	\item\label{3} $\gamma$ is $1$-homogeneous, i.e: for every $c>0$, $\gamma(c\lambda)=c\gamma(\lambda)$ in $\Gamma$.
	
	\item\label{4} $\gamma$ is strictly concave in off-radial direction, i.e: for every $\lambda\in\Gamma$ and $\xi\in\rr^n$ it holds 
	\begin{align*}
	\dfrac{\partial^2\gamma}{\partial \lambda_i\partial\lambda_j}(\lambda)\xi_i\xi_j\leq 0,
	\end{align*}      
	and equality holds if, and only if, $\xi$ is a scalar multiple of $\lambda$. 
	
	\item\label{5} There exists a constant $C>0$ such that 
	\begin{align*}
	\dfrac{d}{ds}\gamma(A+sB)\leq C\mbox{Tr}(B),
	\end{align*} 
	whenever $B$ is $2$-nonnegative matrix with $\lambda(B)\in\Gamma$. Moreover, the inequality is strict unless $B=0$. 
	
	\item\label{6} $\gamma$ vanish at boundary, i.e: There exist a continuous extension of $\gamma$ to $\overline{\Gamma}$ which vanishes identically at $\partial\Gamma$.   
	\end{enumerate}
The main reasons why we consider the above properties are the following:
\begin{itemize}
	\item Property \ref{1} is for having a well defined smooth function. Indeed, it is a well know result, see for instance \cite{glaeser1963fonctions}, that there exist a smooth symmetric function $G:\set{A\in\mbox{Sym}(n):\lambda(A)\in\Gamma}\to\rr$ such that $G(A)=\gamma(\lambda)$ whence $A$ is a diagonal matrix.
	
	\item Property \ref{2} implies that $\gamma$-flow is weakly parabolic. For instance, if we write Equation \eqref{gamma-flow} in local coordinates, property \ref{2} correspond to the ellipticity bound in this chart.   
	
	\item Property \ref{3} implies that the set of solutions of \eqref{gamma-flow} is closed under parabolic scaling.
	
   \item Properties \ref{4} -\ref{6} are imposed to preserve convexity estimates under the flow. For instance, Property \ref{5} preserves $2$-convexity which means that if the principal curvatures of $\Sigma_0$ belongs to $\set{\lambda\in\Gamma:\lambda_i+\lambda_j>0}$ then it is preserved in $\Sigma_t$ for each $t\in(0,T)$. 
   \newline
   Actually, the $\gamma$-flow becomes degenerate when $\gamma$ is close to $\partial\Gamma$. For a detailed explanation of this fact we refer to \cite{BAndrews_JMcCoy_YZheng_2013}.
\end{itemize}

\begin{example}\label{ex1}
The class of functions which satisfy the above properties is quite large, in particular, includes:
\begin{enumerate}
	\item $(S_k)^{\frac1k}$, supported in the Galerkin cone $\Gamma_k=\set{\lambda\in\rr^n:S_l(\lambda)>0, l=1,\ldots,k}$.
	\item  $\lt(\sum\limits_{1\leq i<j\leq n}\dfrac{1}{\lambda_i+\lambda_j}\rt)^{-1}$, supported in $\Gamma=\set{\lambda\in\rr^n:\lambda_i+\lambda_j>0}$.  
	\item Homogeneous one products of the above example. This means, for positive numbers $\alpha_i$ such that $\sum\limits_{i=1}^m\alpha_i=1$, define $\gamma(\lambda)=\prod\limits_{i=1}^mf_i(\lambda)^{\alpha_i}$ over $\Gamma=\bigcap\limits_i^{m}\Gamma_i$, where $f_i$ is one of the examples above.
\end{enumerate}
\end{example}
\begin{remark}
 The quotients of the form $Q_{k,l}=\lt(\dfrac{S_k}{S_l}\rt)^{\frac{1}{k-l}}$ defined in $\Gamma_k$ for $0< l<k\leq n$, satisfy properties \ref{1}-\ref{5} but no \ref{6}.	
\end{remark} 

Recall that in this paper we are interested in translating solitons of Equation \eqref{gamma-flow}, which are defined by 
\begin{align*}
F(x,t)=F_0(x)+tv,\: (x,t)\in \Sigma\times\rr,
\end{align*}
for fixed constant direction $v\in\Sp^n$. We refer to these solutions as $\gamma$-translators for short. Moreover, up to a tangential diffeomorphism,  a $\gamma$-translator can be seen as a hypersurface $\Sigma_0\subset\rr^{n+1}$ such that satisfies a fully nonlinear elliptic equation of the form
\begin{align}\label{gamma-trans}
\gamma(\lambda)=\pI{\nu,v},	
\end{align}	
where $\nu$ and $\lambda$ are the unit normal vector and the principal curvature vector of $\Sigma_0=F_0(\Sigma)$, respectively. 

From the PDE's perspective, these solutions correspond to eternal solutions\footnote{Solutions which exist for all $t\in\rr$.} to Equation \eqref{gamma-flow}. In addition, when $\gamma$ is convex, the authors in \cite{andrews2014convexity} shown that after rescaling a type-II singularities\footnote{Solutions which satisfy $\lim\limits_{t\to T}\sup_{p\in\Sigma}|A(p,t)|\sqrt{T-t}=\infty$, where $T$ stand for maximal time of existence of the $\gamma$-flow.} of the $\gamma$-flow when $\Sigma_0$ is strictly convex, the solution moves under translation by a fixed direction.

The most studied $\gamma$-translators in the literature are when we take $\gamma=H$, see for instance the excellent survey in \cite{hoffmannminimal}. In this case $H$-translators are also minimal hypersurface in the Euclidean space with the conformal metric $e^{\pI{p,v}}\delta$.

The study of graphical $H$-translators in $\rr^3$ starts with the construction of the bowl soliton by Altschuler and Wu in \cite{bowl_1994}. This surface is a entire graph asymptotic to a paraboloid, see \cite{Clutterbuck} for the exact asymptotic behavior. Then, in \cite{haslhofer}, the author show that a $\alpha$-noncollapse and convex $H$-translator must be the bowl soliton. Finally, in \cite{spruck_xiao}, the authors prove that a mean convex $H$-translator is convex. Consequently, $H$-translators which are mean convex are the bowl soliton or pieces from it. We refer the reader to \cite{Paco1} for a classification $H$-translator which are graphs.
\newline

In our setting, we do not have many other examples of $\gamma$-translators in $\rr^{n+1}$. Actually, the author in \cite{Jose} constructs an example of $Q_{n-1}$-translator in $\rr^{n+1}$, which is strictly convex, complete and asymptotic to a cylinder.
\newline
More recently, in \cite{Shati}, the author has shown the existence of bowl type solutions of generic $\alpha$-homogeneous ($\alpha>0$) curvature functions without any convex or concave assumption. These solutions are entire or asymptotic to a round cylinder.
\newline
On the other hand, numerical solutions to Equation \eqref{gamma-trans}, have shown a lack of convexity for rotationally symmetric $\gamma$-translating graphs. From this fact, in this paper we provide a convexity estimate for the smallest principal curvature of a $\gamma$-translator. 

\begin{theorem}\label{thm1.1}
Let $n\geq 3$,  $\alpha,\delta>0$ and  $\Sigma\subset\rr^{n+1}$ be a complete, immersed, two-sided $\gamma$-translator such that its principal curvatures satisfy
\begin{enumerate}
	\item[a)]\label{a} $\lambda\in\Gamma_{\alpha,\delta}=\set{\lambda\in\Gamma: (\delta+1)H\leq\alpha\gamma}$, which is compactly supported in $\Gamma\setminus \mbox{Cyl}_{n-1}$, where 
	\begin{align*}
	\mbox{Cyl}_{j}=\set{\lambda(e_1+\ldots+e_{n-j}):\lambda>0}.
	\end{align*}
	\item[b)]\label{b} There exist a constant $\beta\in (0,1)$ such that  $\lambda_i+\lambda_j\geq \beta H$, for every $1\leq i<j\leq n$. 
\end{enumerate}
Then,  $\lambda_1\geq H-\alpha\gamma$ in $\Sigma$, where $\lambda_1(p)=\min\set{\lambda_i(p):i=1,\ldots,n}$. 
\end{theorem}
We point out the reasons why we restrict under the hypothesis of Theorem \ref{thm1.1} :
\begin{itemize}
	\item The set $\Gamma_{\alpha,\delta}$ is a convex closed subset of $\Gamma$, which is compactly supported in $\Gamma$, i.e: the set 
	\begin{align*}
	\overline{\Gamma}_{\alpha,\delta}\cap\partial B(0,1)
	\end{align*}
	is compact in $\Gamma$. This fact implies an uniform estimate of the second order derivatives of $\gamma$, see Lemma \ref{Eliptic estimate} which is an important part in the proof.

	\item Moreover, since $n\geq 3$, the uniform $2$-convexity implies that $\Gamma\subset\Gamma_{2} $. In particular, we have a uniform bound for the principal curvatures,
	\begin{align}\label{ine1}
	|\lambda_i|^2\leq |A|^2\leq H^2\leq \lt(\frac{\alpha}{\delta+1}\gamma\rt)^2\leq \frac{\alpha^2}{(\delta+1)^2},
	\end{align}
	here we use $\gamma(\lambda)=\pI{\nu,v}\leq 1$.
	\item Finally, the set $\Gamma_{\alpha,\delta}$ and the uniform 2-convexity property are preserved under the $\gamma$-flow, respectively. Indeed, these facts come from the Maximum Principle applied to the equation
	\begin{align*}
	(\partial_t-\Delta_{\gamma})\dfrac{H}{\gamma}=\dfrac{g^{ij}}{\gamma}\dfrac{\partial^2\gamma}{\partial h_{ab}h_{cd}}\nabla_ih_{ab}\nabla_jh_{ab}+\dfrac{2}{\gamma}\pI{\nabla \dfrac{H}{\gamma},\nabla\gamma}_{\gamma},
	\end{align*} 
	where $\Delta_{\gamma}=\frac{\partial \gamma}{\partial h_{ij}}\nabla_i\nabla_j$ and $\pI{X,Y}=\frac{\partial \gamma}{\partial h_{ij}}X_iY_j$. Recall that the second order term is non-negative. For the $2$-convexity, we refer to \cite{Brendle_Huisken_2015} for a detailed explanation of this fact. 
\end{itemize}

\begin{remark}
	It is still an open question which $\alpha$ and $\delta$ are optimal. For instance, one can restrict to 
	\begin{align*}
	\dfrac{\alpha}{\delta+1}=\restri{\dfrac{H}{\gamma}}{\mbox{Cyl}_j}
	\end{align*}
	for some $\mbox{Cyl}_j\subset\Gamma$. 
	\newline
	On the other hand, if there is a Newton-Maclaurin Inequity related to $\gamma$, we could obtain that a closed $\gamma$-translator is totally umbilical. Note that, by a Maximum Principle argument, this $\gamma$-translator could not exist unless it possesses boundary.
\end{remark}

\begin{remark}
	The work of this paper is inspired by \cite{spruck_sun_2019}, where the authors show that a mean convex and uniform $2$-convex translating soliton of the Mean Curvature Flow is convex. 
	\newline
	On one hand, our proof does not hold for translators of the Mean Curvature Flow. The main cause is that the cone $\Gamma_{\alpha,\delta}$ needs to be compactly supported in $\Gamma_1\setminus C_{n-1}$, but for $\gamma=H$, $\Gamma_{\alpha,\delta}=\Gamma_1$. 
	\newline
	On the other hand, the proof given in \cite{spruck_sun_2019} cannot be directly adapted to a concave speed function $\gamma$. The main reason is that the authors use a concave approximation to $\lambda_1$, for which the Maximum Principle does not give suitable information for general speeds $\gamma$.
\end{remark}

\begin{remark}
	A similar estimate, as in Theorem \ref{thm1.1}, was proved in \cite{Lynch1_2020} for a family curvature functions of the form
	\begin{align*}
(1-c)H-c\lt(\sum_{1=i_1<\ldots<i_k=n}\dfrac{1}{\lambda_{i_1}+\ldots\lambda_{i_k}}\rt)^{-1}, \:c\in(0,1).
	\end{align*}
 In contrast with our result, his estimate is from the parabolic and compact perspective.   
\end{remark}

Finally, the aim of the last part of this paper consists in finding more examples of $\gamma$-translators for particular  speeds $\gamma$. Indeed, we find two different examples of $\sqrt{S_2}$-translator in $\rr^3$.

\begin{theorem}\label{S_2-translators}
	The surfaces 
	\begin{align*}
	\Sigma_1=\set{(x,u(|x|))\in\rr^3:x\in\rr^2},\mbox{ for } u(|x|)=\int\limits_0^{|x|}\sqrt{e^{r^2}-1}dr, |x|=\sqrt{x_1^2+x_2^2+x_3^2},
	\end{align*}
	and
	\begin{align*}
	&\Sigma_2=\set{(r(z)\cos(\theta),r(z)\sin(\theta),z)\in\rr^{3}: \theta\in[0, 2\pi),\:z\in\rr },
	\\
	&1+z=\int\limits_{1}^{r(z)}\sqrt{e^{s^2}-1}ds \mbox{ with }r(0)=1,
	\end{align*}
	are $\sqrt{S_2}$-translators in $\rr^3$. In addition, $\Sigma_1$ is strictly convex and entire. On the other hand, $\Sigma_2$ is complete and satisfies $H<0$ and $K>0$.  
\end{theorem} 

\begin{figure} \label{fig}
	\centering
	\includegraphics[width=0.25\columnwidth]{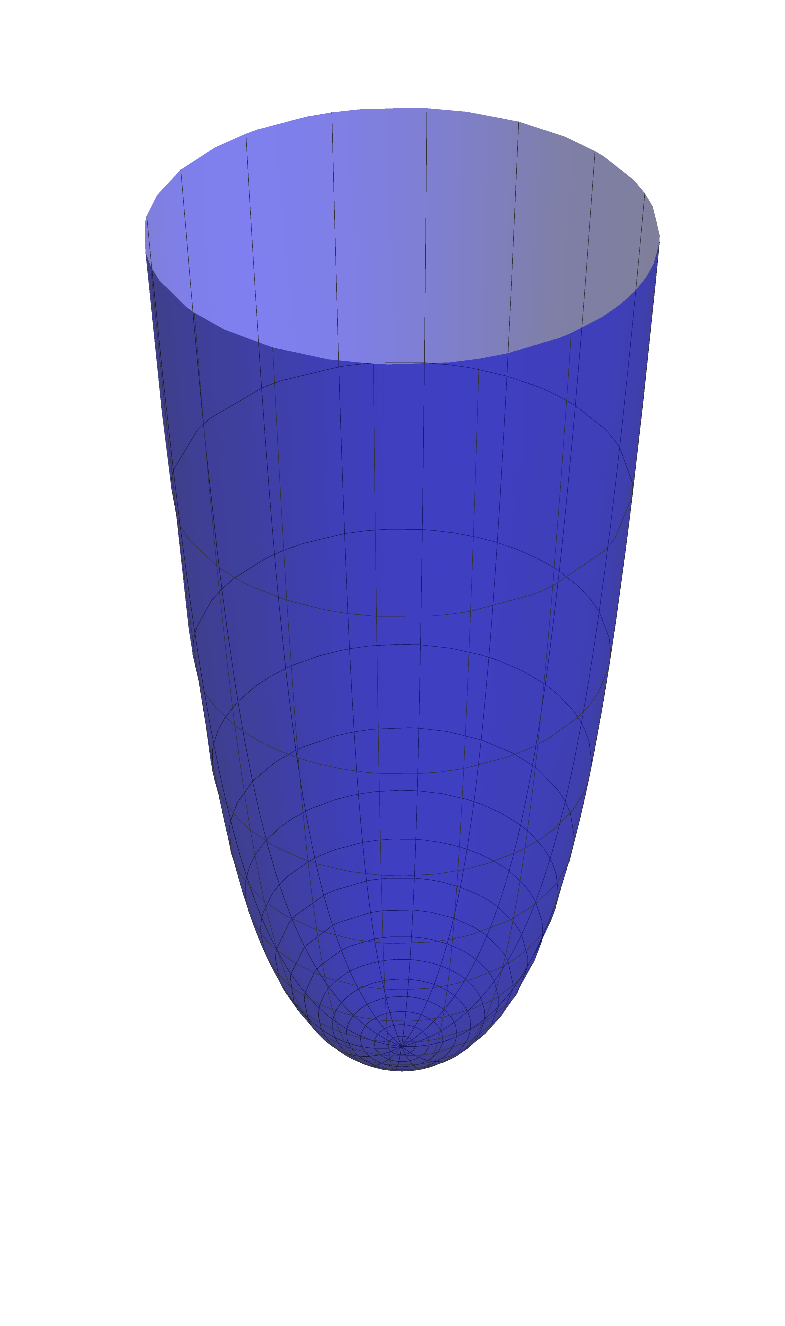}
	\vskip-10mm
	\caption{ The $\sqrt{S_2}$-translator $\Sigma_1$ in $\rr^3$.}
\end{figure}

\begin{figure} \label{fig}
	\centering
	\includegraphics[width=0.40\columnwidth]{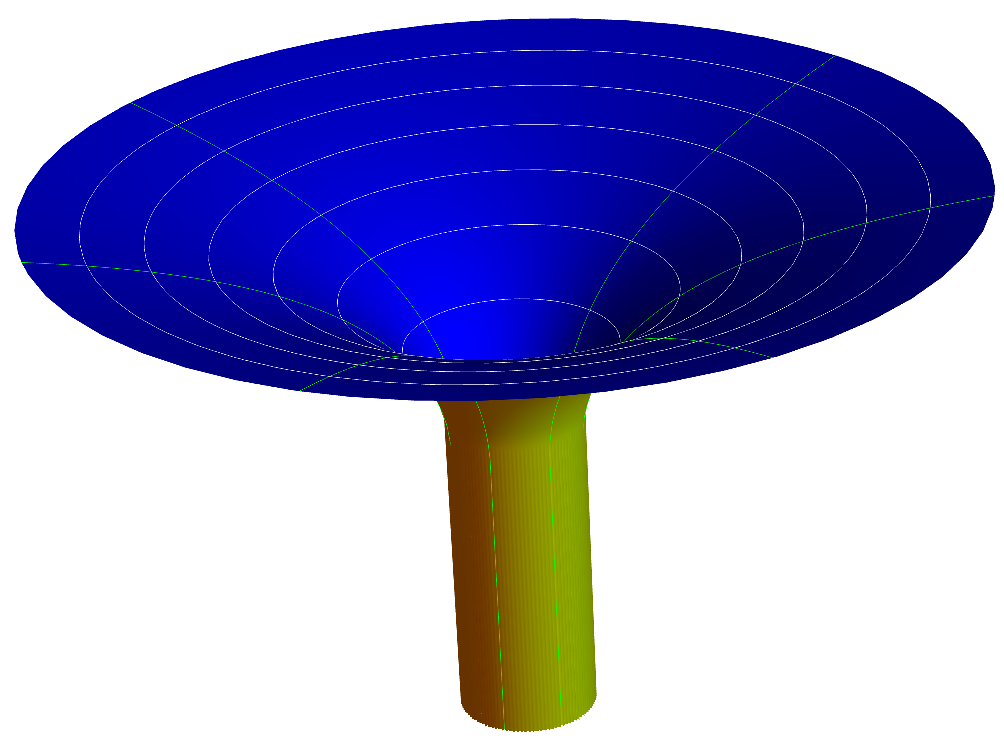}
	\caption{ The $\sqrt{S_2}$-translator $\Sigma_2$ in $\rr^3$.}
\end{figure}

In addition, we find convex examples of cylindrical type for the function $\lt(\sum\limits_{1\leq i<j\leq n}\frac{1}{\lambda_i+\lambda_j}\rt)^{-1}$.
\begin{theorem}\label{2-convexity}
Let $n\in\set{3,\ldots,6}$. There exist a convex complete  $\lt(\sum\limits_{1\leq i<j\leq n}\frac{1}{\lambda_i+\lambda_j}\rt)^{-1}$-translators in $\rr^{n+1}$ of the form
 \begin{align*}
 \set{(x,u(|x|))\in\rr^{n+1}:x\in B\lt(0,r_n\rt)},
 \end{align*}
where $r_n=\frac{8}{n^2+n+2}$, and $u\to\infty$ when $|x|\to\frac{8}{n^2+n+2}$.
\end{theorem}
\begin{remark}
It is not hard to see that the uniqueness and the non-existences theorems from \cite{Jose}, can be applied to the examples of Theorem \ref{2-convexity}. More precisely, these results state that a $\gamma$-translator which is a graph and  asymptotic to a cylinder must be rotationally symmetric. In addition, if one has a convex  $\gamma$-translating graph defined in a precompact domain $\Omega\subset\rr^n$ such that $\mbox{diam}(\Omega)<2r_n$ or $B(x,r_n)\subset\Omega$ will not exist.   
\end{remark}

The structure of the paper is summarized as follows: In Section \ref{sec2} we state some notation and properties of the function $\gamma$ in the principal curvatures and the second fundamental form matrix $A=(h_{ij})$ of $\Sigma$, respectively. In Section \ref{sec3} we proove Theorem \ref{thm1.1}. In Section \ref{sec4} we prove Theorems \ref{S_2-translators} and\ref{2-convexity}.
\newline

\underline{\textbf{Acknowledgment:}} The author would like to thank F. Martín and M. Sáez for bringing this problem to his attention and for all the support they have provided. Furthermore, the author would like to thank S. Rengaswami for his advice and encouragement about this topic.

\section{Preliminars}\label{sec2}

In the following, we will consider a $\gamma$-translator $\Sigma$ which evolves under translation in the $x_{n+1}$-axis.

It will be convenient to fix notation for derivatives of the symmetric functions in the principal curvatures of $\Sigma$. In fact,  we will abuse the notation by setting  $\varphi(A)=\varphi(\lambda(A))$ for a symmetric function $\varphi$ and symmetric matrix $A=(h_{ij})$ such that $\lambda(A)\in\Gamma$. Then, we write
\begin{align*}
\dot{\varphi}^{ab}(A)=\dfrac{\partial\varphi}{\partial h_{ab}}(A),\:\dot{\varphi}^a(\lambda)=\dfrac{\partial\varphi}{\partial\lambda_a}(\lambda),\ddot{\varphi}^{ab,cd}(A)=\dfrac{\partial^2\varphi}{\partial h_{ab}\partial h_{cd}}(A),\:\ddot{\varphi}^{ab}(\lambda)=\dfrac{\partial^2\varphi}{\partial\lambda_a\partial\lambda_b}(\lambda).
\end{align*}
Furthermore, if $\varphi$ is $1$-homogeneous, then
\begin{align*}
\dot{\varphi}^{ab}(A)=\delta_{ab}\dot{\varphi}^a(\lambda),
\end{align*}  
for any diagonal matrix $A$. If in addition the eigenvalues are simple, $\lambda_1<\ldots<\lambda_n$, then we can write
\begin{align*}
\ddot{\varphi}^{ab,cd}(A)T_{ab}T_{cd}=\ddot{\varphi}^{ab}(\lambda)T_{aa}T_{bb}+2\sum_{a<b}\dfrac{\dot{\varphi}^b(\lambda)-\dot{\varphi}^{a}(\lambda)}{\lambda_b-\lambda_a}|T_{ab}|^2,
\end{align*}
for every symmetric matrix $T_{ab}$.

	\begin{lemma}\label{Eq}
		Let $\Sigma$ be a  $\gamma$-translator. Then, we have the following equations at $p\in \Sigma$
		\begin{align}
		&\Delta_{\gamma}h_{ij}+\ddot{\gamma}^{ab,cd}\nabla_ih_{ab}\nabla_jh_{cd}+|A|^2_\gamma h_{ij}+\pI{\nabla h_{ij},e_{n+1}}=0.
		\\\label{varphi Eq}
		&\Delta_\gamma\varphi +\lt(\dot{\varphi}^{ij}\ddot{\gamma}^{ab,cd}-\dot{\gamma}^{ij}\ddot{\varphi}^{ab,cd}\rt)\nabla_ih_{ab}\nabla_jh_{cd}+|A|_\gamma^2\varphi+\pI{\nabla \varphi,e_{n+1}}=0,
		\\ \label{gamma Eq}
		&\Delta_\gamma \gamma+|A|_\gamma^2\gamma+\pI{\nabla \gamma,e_{n+1}}=0,
		\end{align}
		where 
		\begin{align*}
		\Delta_\gamma=\dot{\gamma}^{ab}\nabla_a\nabla_b,\:\pI{X,Y}_{\gamma}=\dot{\gamma}^{ab}X_aY_b\mbox{ and }|A|_{\gamma}^2=\dot{\gamma}^{ab}h_{ai}h_{ib}. 
		\end{align*}
	\end{lemma}
	\begin{proof}
		By choosing normal coordinates at $F(x,t)=F_0(x)+te_{n+1}$, it follows that $\partial_th_{ij}=\pI{\nabla h_{ij},e_{n+1}}$. Therefore, all the above equations follow by replacing the time derivative in the evolution equations of these quantities under the $\gamma$-flow, with the gradient projected onto the translation vector $e_{n+1}$. We refer  the reader to\cite{huisken_polden_1996} for the evolution equations under general geometric flows.
	\end{proof}

\begin{lemma}\label{Eliptic estimate}
	Let $\Gamma'$ be a symmetric closed cone such that is compactly supported in \newline $\Gamma\setminus\mbox{Cyl}_{n-1}$, where 
	\begin{align*}
	\mbox{Cyl}_{j}=\set{\lambda(e_1+\ldots+e_{n-j}):\lambda>0}.
	\end{align*}
	Then, there exist a constant $0<C=C(n,\gamma,\Gamma')$ such that 
	\begin{align*}
	&C^{-1}\delta_{ab}\dot{\gamma}^{a}(\lambda)\leq \dot{\gamma}^{ab}(A)\leq C\delta_{ab}\dot{\gamma}^i(\lambda),
	\\
	&\sum_{i}\ddot{\gamma}^{ab,cd}(A)T_{iab}T_{icd}\leq -C\dfrac{|T|^2}{H},
	\end{align*}
	where $A$ is a diagonal matrix with eigenvalues $\lambda\in\Gamma'$ and $T_{iab}$ is a totally symmetric tensor.
\end{lemma}

\begin{proof}
	We refer the reader to \cite{Lynch_2020} for the proof. 
\end{proof}

\section{Proof of Theorem \ref{thm1.1}}

We consider on $\Sigma$ the function 
\begin{align*}
 f(p)=\dfrac{\gamma(\lambda)}{S_{1,1}(\lambda)},
\end{align*}
where $S_{1,1}=H-\lambda_1$ and $\lambda_1(p)=\min\limits_{i=1,\ldots,n}\set{\lambda_i(p)}$. Therefore, the convexity estimate of Theorem \ref{thm1.1} can be written as  
\begin{align}\label{Ine1}
f(p)\geq \frac{1}{\alpha},\mbox{ for every }p\in\Sigma. 
\end{align}	

The proof will be by contradiction, which means that we assume 
\begin{align*}
\inf\limits_{\Sigma} f<\alpha^{-1}.
\end{align*}
 In particular, this assumption together with $\lambda\in\Gamma_{\alpha,\delta}=\set{\lambda\in\Gamma:(\delta+1)H\leq \alpha \gamma}$ gives an estimate of the form
\begin{align}\label{control}
\lambda_1< H-\alpha \gamma\leq -\delta H
\end{align}
for $p_N\in\Sigma$ such that $f(p_N)\to\inf\limits_\Sigma f$. This is a key fact in our proof, since it permits to control some curvature terms.

Before continuing with the proof we explain the steps of it. First, we study the case when $f$ attains  its minimum at an interior point. This would imply that $\gamma$ vanishes in all $\Sigma$, which contradicts the fact that $\gamma>0$ in $\Sigma$ . Consequently, the only possibility is that the infimum is attained by a sequence $p_N\in\Sigma\in$ such that $|p_N|\to\infty$. To work with this setting, we apply the Omori-Yau Maximum Principle for the sequence $\Sigma_N=\Sigma-\set{p_N}$ of $\gamma$-translators. By a compactness argument, we will get a limit $\Sigma_\infty$, which is a $\gamma$-translator, to restart our argument for which the infimum cannot be attained.

Now, we calculate some equations related to $f$. For doing this, we use an orthonormal frame of principal directions $\set{e_i}\subset T_{p}\Sigma$ of $\Sigma$ at $p$. 
\newline
Then, by the equations in Lemma \ref{Eq}, it follows that
\begin{align}\label{grad f}
\nabla_a f=&\dfrac{S_{1,1}\nabla_a\gamma-\gamma\nabla_a S_{1,1}}{(S_{1,1})^2},
\\ \label{Delta f}
\Delta_\gamma f=&\dfrac{f}{S_{1,1}}\lt(\dot{S_{1,1}}^{ij}\ddot{\gamma}^{ab,cd}-\dot{\gamma}^{ij}\ddot{S_{1,1}}^{ab,cd}\rt)\nabla_ih_{ab}\nabla_jh_{cd}
\\ \notag
&-\nabla_{n+1}f-2\pI{\nabla f, \dfrac{\nabla S_{1,1}}{S_{1,1}}}_\gamma.
\end{align}

We divide the proof into several claims.  
\begin{claim}\label{Claim1}
Assume that there exist $p_0\in\Sigma$, such that $f(p_0)=\inf_\Sigma f$. Then, $\abs{A}_\gamma^2\gamma=0$ in $\Sigma$. In consequence, $f$ cannot achieve its minimum at an interior point of $\Sigma$.
\end{claim}

\begin{proof}\label{sec3}
	Recall that by \eqref{control}, $\lambda_1(p_0)<0$. Rewriting  Equation \eqref{Delta f}, we see that
	\begin{align}\label{Delta f1}
	\Delta_\gamma f+\nabla_{n+1}f+2\pI{\nabla f,\dfrac{\nabla S_{1,1}}{S_{1,1}}}_\gamma=\dfrac{f}{S_{1,1}}\lt(\dot{S}_{1,1}^{ij}\ddot{\gamma}^{ab,cd}-\dot{\gamma}^{ij}\ddot{S}_{1,1}^{ab,cd}\rt)\nabla_ih_{ab}\nabla_jh_{cd}\leq 0,
	\end{align}
	for the sign we used the concavity of $\gamma$, the convexity of $S_{1,1}$, and the fact that $S_{1,1}$ and $\gamma$ are increasing in each variable. 
	\newline
	Note that the left-hand side in Equation \eqref{Delta f1} is proper, $1$-homogeneous and locally uniformly elliptic as $\lambda\in\Gamma_{\alpha,\delta}$. Therefore, since $f(p_0)$ is a minimum, the Strong Maximum Principle (see for instance  Theorem 4.2 in \cite{goffi2021note})
	implies that the function $f(p)$ is constant  in $\Sigma$. 
	\newline
	Consequently, the second order term in \eqref{Delta f}
	\begin{align*}
	\dot{S}_{1,1}^{i,j}\ddot{\gamma}^{ab,cd}\nabla_ih_{ab}\nabla_jh_{cd}=(1-\delta_{1i})\ddot{\gamma}^{ab,cd}\nabla_ih_{ab}\nabla_ih_{cd},
	\end{align*}
	must vanish at $\Sigma$. Therefore, by Lemma \ref{Eliptic estimate}, we have that $\nabla _ih_{ab}=0$ for $i>1$. 
	\newline
	Then, by the Codazzi Equations, it holds that $|\nabla S_{1,1}|=0$ on $\Sigma$. 
	In particular, by Equation \eqref{grad f}, we have $|\nabla\gamma|=0$ on $\Sigma$. Therefore,  $\gamma$ is constant in $\Sigma$.
	\newline
	Finally, by  Equation \eqref{gamma Eq}, it follows that  $|A|_\gamma^2\gamma=0$ which cannot occur since at $\lambda(p_0)\in\Gamma_{\alpha,\delta}$ for which the functions $\gamma$ and $|A|_\gamma^2$ are positive. 
\end{proof}

Now we focus on the case when $\inf\limits_{\Sigma}f(p)\in\lt(0,\alpha^{-1}\rt)$ is attained at ``infinity". We use the Omori-Yau Maximum Principle, which gives a sequence $p_N\in \Sigma$ such that $|p_N|\to\infty$, and the following equations hold
\begin{align}\label{OYMP}
f(p_N)\to \inf_\Sigma f(p),\:|\nabla f(p_N)|<\frac{1}{N}\mbox{ and }\:\Delta_\gamma f(p_N)\geq -\frac{1}{N}.
\end{align}
 We also consider the sequence of $\gamma$-translators given by
\begin{align*}
\Sigma_N=\Sigma-p_N.
\end{align*} 
Note that the Second Fundamental form of each $\Sigma_N$ is uniformly bounded by the same constant since $\Sigma_N$ is a translation of $\Sigma$ by $p_N$. In addition,  we note that $0\in \Sigma_N$ for each $N$. 
\newline
Therefore, by compactness (see \cite{perez_ros_2002} for details), we can subtract a sub-sequence $\Sigma_N'$ of $\Sigma_N$ and a $\gamma$-translator $\Sigma_\infty'$ such that $\Sigma_N'\to \Sigma_\infty'$ uniformly smoothly on compacts subset of $\Sigma_\infty'$.
\newline
 Finally, we denote by $\Sigma_{\infty}$ the connected component of $\Sigma'_{\infty}$ which contains the point $0$.   

\begin{claim}\label{Claim2}
	At $0\in \Sigma_{\infty}'$, the function $\gamma$ vanishes.
\end{claim}

\begin{proof}
	If it is not the case, then $\gamma>0$ at $0\in \Sigma'_{\infty}$. But since $f(0)=\inf\limits_{\Sigma}f$ we get a contradiction with Claim \ref{Claim1}.
\end{proof}

\begin{claim}\label{Claim3}
All the principal curvatures vanish at $0\in\Sigma_\infty'$. Moreover, we have the following estimates at $p_N$,
\begin{align*}
&f\geq \dfrac{1+\delta}{2\alpha},
\\
&\dfrac{\lambda_i}{S_{1,1}}\geq \frac{\beta}{2},
\\
&\dfrac{|\lambda_1|}{S_{1,1}}>\dfrac{\delta}{2}.
\end{align*}
\end{claim}

\begin{proof}
The first part follows from  inequality \eqref{ine1} and the fact that $\gamma(p_N)\to 0$ as $N\to\infty$. 
\newline 
Remember that by \eqref{control} and $N$ big enough, we have $\lambda_1(p_N)<0$. Then, since $n\geq 3$, at $p_N$ it follows
\begin{align*}
H+\lambda_1\geq (\lambda_3+\lambda_1)+(\lambda_2+\lambda_1),
\end{align*}
which is positive by the $2$-convexity. In particular, we have 
\begin{align*}
H-\lambda_1\leq 2H.
\end{align*}
Therefore, it follows that 
\begin{align*}
f(p_N)=\dfrac{\gamma}{H-\lambda_1}\geq\dfrac{\gamma}{2H}\geq\dfrac{1+\delta}{2\alpha},
\end{align*}
in the last inequality we use $\lambda\in\Gamma_{\alpha,\delta}$.
\newline
Moreover, by the two uniform convexity, we obtain
\begin{align*}
\dfrac{\lambda_i}{H-\lambda_1}=\dfrac{\lambda_i+\lambda_1-\lambda_1}{H-\lambda_1}\geq\dfrac{\beta H}{H-\lambda_1}\geq \frac{\beta}{2}, 
\end{align*}
at $p_N$ and $i>1$. In the first inequality we use the $2$-uniform convexity $\lambda_i+\lambda_1\geq \beta H$ and the fact that $\lambda_1(p_n)<0$. 
\newline
Finally, for the last term, we have
 \begin{align*}
 \dfrac{|\lambda_1|}{H-\lambda_1}=\dfrac{-\lambda_1}{H-\lambda_1}>\dfrac{\alpha \gamma-H}{H-\lambda_1}\geq\dfrac{\delta H}{H-\lambda_1}\geq \dfrac{\delta}{2},
 \end{align*}
 in the second inequality we use Equation \eqref{control}. 
\end{proof}

Recall equation \eqref{grad f},
\begin{align}\notag
\nabla f&=\dfrac{\nabla \gamma}{S_{1,1}}-f\dfrac{\nabla S_{1,1}}{S_{1,1}},
\\ \label{gradf1}
&=\dfrac{\lambda_i\pI{e_i,e_{n+1}}}{S_{1,1}}e_i-f\dfrac{\nabla S_{1,1}}{S_{1,1}}.
\end{align}
By Claim \ref{Claim3}, the term
\begin{align*}
\frac{\nabla_i \gamma}{S_{1,1}}=\frac{\lambda_i\pI{e_i,e_{n+1}}}{S_{1,1}}e_i
\end{align*}
is uniformly bounded at $p_N$. In addition, since $\nabla f\to 0$ at $N\to\infty$, it follows that the term $\dfrac{\nabla S_{1,1}}{S_{1,1}}(p_N)$ is also bounded.

\begin{claim}
	The term 
	\begin{align*}
	\dfrac{\nabla S_{1,1}}{S_{1,1}}(p_N)\to 0
	\end{align*}
	as $N\to\infty$. 
\end{claim}

\begin{proof}
	We evaluate \eqref{Delta f} at $p_N$. Then, by equation \eqref{OYMP}, we note
	\begin{align*}
	\dfrac{-C}{N}\leq\Delta_\gamma f+\nabla_{n+1}f+2\pI{\nabla f, \dfrac{\nabla S_{1,1}}{S_{1,1}}}_\gamma=f\frac{\lt(\dot{S}_{1,1}^{ij}\ddot{\gamma}^{ab,cd}-\dot{\gamma}^{ij}\ddot{S}_{1,1}^{ab,cd}\rt)}{S_{1,1}}\nabla_ih_{ab}\nabla_jh_{cd}\leq 0.
\end{align*}
	Then, since $\nabla f\to 0$, $f$ and the term $\dfrac{\nabla S_{1,1}}{S_{1,1}}$ are bounded from below by a positive constant and from above as well, it follows that 
	\begin{align}\label{lim}
	\dfrac{1}{S_{1,1}}\lt(\dot{S}_{1,1}^{ij}\ddot{\gamma}^{ab,cd}-\dot{\gamma}^{ij}\ddot{S}_{1,1}^{ab,cd}\rt)\nabla_ih_{ab}\nabla_jh_{cd}\to 0,
	\end{align}
	as $N\to\infty$. 
	\newline
	Secondly, we note that if $\lambda_a=\lambda_b$ at $p_N$, then $\nabla h_{ab}(p_N)=0$. Indeed, by taking tensorial derivatives we have
	\begin{align*}
	0 = e_i(h_{ab})=\nabla_i h_{ab} - A(\nabla_ie_a,e_b) -A(e_a,\nabla_ie_b)= h_{ab,i} + \Gamma_{ia}^b(\lambda_a-\lambda_b)= h_{ab,i}.
	\end{align*}
    Therefore, we may assume that $\lambda_a\neq\lambda_b$ at $p_N$. Then, for $i>1$ and using $\lambda_1(p_N)<0$ with Lemma \ref{Eliptic estimate}, we obtain
	\begin{align*}
	\lt(\dfrac{\nabla_i h_{ab}}{S_{1,1}}\rt)^2&\leq \dfrac{\abs{\nabla_i h_{ab}}^2}{S_{1,1}H}\leq\dfrac{-1}{CS_{1,1}}\sum_{i>1}\ddot{\gamma}^{ab,cd}\nabla_ih_{ab}\nabla_i h_{cd}.
\end{align*}	
Then, by noting that $\frac{\partial S_{1,1}}{\partial \lambda_i}=(1-\delta_{1i})$ and adding the term $\frac{\dot{\gamma}^{ij}\ddot{S}_{1,1}^{ab,cd}}{S_{1,1}}$ in the last inequality, it yields 
\begin{align*}
	\dfrac{\abs{\nabla_i h_{ab}}^2}{(S_{1,1})^2}\leq\dfrac{-1}{CS_{1,1}}\lt(\dot{S}_{1,1}^{ij}\ddot{\gamma}^{ab,cd}-\dot{\gamma}^{ij}\ddot{S}_{1,1}^{ab,cd}\rt)\nabla_ih_{ab}\nabla_jh_{cd}
	\end{align*}
	which goes to $0$ by Equation \eqref{lim}. 
	\newline
	Consequently, since at $p_N$ we have 
	\begin{align*}
	\abs{\dfrac{\nabla_iS_{1,1}}{S_{1,1}}}=(1-\delta_{1a})\abs{\dfrac{\nabla_ih_{aa}}{S_{1,1}}},
	\end{align*}
	then by changing the indices with Codazzi Equations it holds that $\frac{\nabla S_{1,1}}{S_{1,1}}\to 0$ as $N\to\infty$.  
\end{proof}
 We note that  the last claim together with Equation \eqref{gradf1} implies that
 \begin{align*}
 \lt|\dfrac{\nabla_i \gamma}{S_{1,1}}\rt|=\abs{\dfrac{\lambda_i}{S_{1,1}}}\abs{\pI{e_i,e_{n+1}}}\to 0 \mbox{ as }N\to\infty.
 \end{align*}
 Finally, by Claim \ref{Claim3}, which states a lower bound for each $\frac{\lambda_i}{S_{1,1}}(p_N)$, we have that $\pI{e_i,e_{n+1}}\to0$, as well as $\pI{\nu,e_{n+1}}\to 0$ when $N\to\infty$. In particular, we get that $e_{n+1}=\vec{0}$ at $0\in\Sigma_{\infty}'$, which is a contradiction, finalizing the proof of Theorem \ref{thm1.1}.

\section{Examples}\label{sec4}

In this section we consider speed functions  $\gamma:\Gamma\to\rr$ of the form 
\begin{align*}
\lt(S_k(\lambda)\rt)^{\frac1k}\mbox{ and }\lt(\sum\limits_{1\leq i<j\leq n}\frac{1}{\lambda_i+\lambda_j}\rt)^{-1},
\end{align*}
and study the existence of rotationally symmetric $\gamma$-translators in $\rr^{n+1}$. Recall that a rotationally symmetric graph  $\Sigma\subset\rr^{n+1}$ is given by 
\begin{align*}
\Sigma=\set{(r\theta,u(r))\in\rr^{n+1}:\theta\in\Sp^{n-1},\:r\in[0,R) },
\end{align*} 
where $r=\abs{x}$ and $u:[0,R)\to \rr$ is smooth function satisfying  $u(0)\footnote{The first condition is not neccesary since $\Sigma$ is invariant under translations in the $x_{n+1}$-axis.}=\dot{u}(0)=0$. 
\newline
In addition, the principal curvatures of $\Sigma$ are given by 
\begin{align*}
\lambda_1=\dfrac{\ddot{u}}{(1+\dot{u}^2)^{\frac32}}\mbox{ and }\lambda_2=\ldots=\lambda_{n}=\dfrac{\dot{u}}{r(1+\dot{u}^2)^{\frac12}},
\end{align*}
Therefore, the functions $\gamma$ can be written by 
\begin{align*}
S_k&=\frac{1}{k}\binom{n-1}{k-1}\lt(\dfrac{\dot{u}}{r\sqrt{1+\dot{u}^2}}\rt)^{k-1}\lt[\dfrac{(n-k)\dot{u}}{r\sqrt{1+\dot{u}^2}} +\dfrac{k\ddot{u}}{(1+ \dot{u}^2)^{\frac32}}\rt],
\\
\sum\limits_{1\leq i<j\leq n}\frac{1}{\lambda_i+\lambda_j}&=\dfrac{n-1}{\frac{\ddot{u}}{(1+\dot{u}^2)^{\frac32}}+\frac{\dot{u}}{r\sqrt{1+\dot{u}^2}}}+\dfrac{\binom{n-1}{2}}{2}\dfrac{r\sqrt{1+\dot{u}^2}}{\dot{u}}.
\end{align*}

\begin{proposition}
	The ODE of rotationally symmetric $(S_k)^{\frac1k}$-translators is given by
\begin{align*}
\begin{cases}
\ddot{u}=F_{k,n}(r,\dot{u})
\\
u(0)=\dot{u}(0)=0. 
\end{cases},
\end{align*}
where  $F_{k,n}(x,y)=\frac{(1+y^2)}{k}\lt( \frac{k}{\binom{n-1}{k-1}}\lt(\frac{x}{y} \rt)^{k-1} -(n-k)\lt(\frac{y}{x}\rt)\rt)$. 
\newline
The ODE of rotationally symmetric $\lt(\sum\limits_{1\leq i<j\leq n}\frac{1}{\lambda_i+\lambda_j}\rt)^{-1}$-translators is given by
\begin{align*}
\begin{cases}
	\ddot{u}=G_n(r,\dot{u})
	\\
	u(0)=\dot{u}(0)=0. 
\end{cases},
\end{align*}
where $G_n(x,y)=\frac{y}{x}(1+y^2)\frac{nx-y}{y-\frac{x}{2}\binom{n-1}{2}}$.
\end{proposition}

\begin{proof}
It is a straightforward computation since $\gamma=\pI{\nu,e_{n+1}}=\frac{1}{\sqrt{1+(\dot{u})^2}}$. 	
\end{proof}
\begin{remark}
	Note that both equations can be seen as a first order ODE by setting $\dot{v}(r)=u(r)$ with $v(0)=0$. Therefore, the existences of solution will follow by simple integration. On the other hand, we find barriers that lead to extensibility of the solution up to a blow-up point and the convexity property.
\end{remark}

\subsection{Existence of $(S_k)^\frac1k$-translators}
\.\newline
In this subsection we study the existence of sub and super solutions to equation 
\begin{align}\label{v=Fkn}
\begin{cases}
\dot{v}=F_{k,n}(r,v)=\frac{v}{r}(1+v^2)\lt( \frac{1}{\binom{n-1}{k-1}}\lt(\dfrac{r}{v} \rt)^{k} -\dfrac{(n-k)}{k}\rt),
\\
v(0)=0
\end{cases}
\end{align}

\begin{proposition}\label{v_+,a}
Let $a\geq 0$. The function
\begin{align*}
v_{\pm,a}(r)=\pm\sqrt{\lt(e^{r^2+a}-1\rt)}, 
\end{align*}
satifies $\dot{v}=F_{2,2}(r,v)$ for $r\geq 0$.
\end{proposition}
\begin{proof}
It is a straightforward calculation. 
\end{proof}
\begin{proposition}\label{Barriers Fkn}
	The functions
	\begin{align*}
	v_1(r)=\lt(\dfrac{k}{n\binom{n-1}{k-1}}\rt)^\frac1kr\mbox{ and }v_2(r)=\binom{n-1}{k}^{-\frac1k}r,
	\end{align*}
	are sub-solution for $k\in\nn$ and super-solution for $k\in \set{2,\ldots,n-1}$ to Equation \eqref{v=Fkn}, respectively.
\end{proposition}
\begin{proof}
Let $v(r)=\alpha r$. Then, we have
	\begin{align*}
	\dot{v}-F_{k,n}(r,v)=\alpha\lt[1-(1+v^2)\lt(\dfrac{1}{\alpha^k\binom{n-1}{k-1}}-\dfrac{n-k}{k}    \rt) \rt].
	\end{align*}
	By choosing $\alpha=\binom{n-1}{k}^{-\frac1k}$, the term 
	\begin{align*}
	\dfrac{1}{\alpha^k\binom{n-1}{k-1}}-\dfrac{n-k}{k} =0.
	\end{align*}
	Therefore, the function $v_2$ is a super-solution to Equation \eqref{v=Fkn}.
	\newline
	In addition, we note that $\alpha=\lt(\frac{k}{n\binom{n-1}{k-1}}\rt)^\frac1k$ gives that
	\begin{align*}
	\dfrac{1}{\alpha^k\binom{n-1}{k-1}}-\dfrac{n-k}{k}\geq 1.
	\end{align*}
In particular, $\dot{v_1}-F_{k,n}(r,v_1)\leq 0$, which means that $v_1$ is  a sub-solution to Equation \eqref{v=Fkn}. 
\end{proof}

\begin{remark}
By Proposition \ref{Barriers Fkn} any solution $v(r)$ to Equation \eqref{v=Fkn}  satisfies $v_1(r)\leq v(r)\leq v_2(r)$ for $r\geq 0$.  In particular, the term in Equation \eqref{v=Fkn} satisfies   
\begin{align*}
0\leq\dfrac{1}{\binom{n-1}{k-1}}\lt(\frac{r}{v} \rt)^{k} -\frac{(n-k)}{k}\leq 1.
\end{align*}
In consequence, any solution to $\ddot{u}=F_{k,n}(r,\dot{u})$ with $u(0)=\dot{u}(0)=0$ will be  convex. 
\end{remark}

\begin{remark}
	The function 
	\begin{align*}
	v_{3}(r)=\frac{\lt(\frac{k}{n\binom{n-1}{k-1}}\rt)^\frac{1}{k}r}{\sqrt{1-\lt(\frac{k}{n\binom{n-1}{k-1}}\rt)^\frac2k r^2}},
	\end{align*}
	which satisfies the equation $\dot{v}_{3}=\frac{v_{3}}{r}(1+v^2_{3})$, is also a super solution to Equation \eqref{v=Fkn}. Indeed, it follows
	 \begin{align*}
	 \dot{v}_3-F_{k,n}(r,v_3)=\frac{n}{k}\frac{v_3}{r}(1+v_3^2)\left(1-\sqrt{1-\lt(\frac{k}{n\binom{n-1}{k-1}}\rt)^\frac2k r^2}\right)\geq 0. 
	 \end{align*} 
	Consequently, if there are any sub solution with a singularity at $r=\lt(\frac{k}{n\binom{n-1}{k-1}}\rt)^\frac{-1}{k}$, then a solution to Equation \eqref{v=Fkn} would have a blow up at $r=\lt(\frac{k}{n\binom{n-1}{k-1}}\rt)^\frac{-1}{k}$. Note that the $\sqrt{S_2}$-translator $\Sigma_1$ is entire, and thus would not poses a sub solution with this characteristic. In fact, this phenomena agree  with the results given in \cite{Shati}. 
\end{remark}

Now we prove Theorem \ref{S_2-translators}.
\begin{proof}[Proof of Theorem \ref{S_2-translators}:]
	The existence of the  $\sqrt{S_2}$-translator 
\begin{align*}
\Sigma_1=\set{(x,u(|x|))\in\rr^3:x\in\rr^2},
\end{align*}
where $u(|x|)=\bigintssss\limits_0^{|x|}v_{+,0}\:dr$ with  $|x|=\sqrt{x_1^2+x_2^2+x_3^2}$ comes by simple integration of the function $v_{+,0}$ found in Proposition \ref{v_+,a}. 
\newline
On the other hand, for surface of the form
	\begin{align*}
	\Sigma_2=\set{\lt(r(z)\cos(\theta),r(z)\sin(\theta),z\rt):(\theta,z)\in[0,2\pi]\times\rr},
	\end{align*}
	where $r(z)$ is a function to be found, we note that the principal curvatures are
	\begin{align*}
	\lambda_1=\dfrac{r''}{(1+r'^2)^{\frac32}}\mbox{ and }\lambda_2=\dfrac{-1}{r\sqrt{1+r'^2}}.
	\end{align*}
	Then, equation $(S_2)^{\frac12}=\pI{\nu,e_3}$ reads as 
	\begin{align*}
	r''=-(1+r'^2)rr'^2\Leftrightarrow ff'=-(1+f^2)rf^2,
	\end{align*}
	where  $\dot{r}=f(r)$. Consequently, $f'=0$ or
    $f'=-(1+f^2)rf$, which have as a solution $f=c$ or $f=\dfrac{1}{\sqrt{e^{r^2-2a}-1}}$ for any constant $a\in\rr$. 
    \newline
    Finally, by taking $a=0$ and an integration procedure, we obtain  
    \begin{align*}
  1+z=\int\limits_{1}^{r(z)}\sqrt{e^{s^2}-1}ds \mbox{ with }r(0)=1
    \end{align*}
	\newline
	In particular, it is not hard to see that $\Sigma_2$, satisfies $H<0$ and $K>0$.
\end{proof}

\subsection{Existence of $\lt(\sum\limits_{1\leq i<j\leq n}\frac{1}{\lambda_i+\lambda_j}\rt)^{-1}$-translators}
\.\newline
In this subsection we study the existence of sub and super solutions to equation 
\begin{align}\label{w=Gn}
\begin{cases}
\dot{w}=G_n(r,w)=\frac{v}{r}(1+v^2)\frac{n-\frac{w}{r}}{\frac{w}{r}-\frac{n^2-3n+2}{4}},
\\
w(0)=0.
\end{cases}
\end{align}

  \begin{proposition}\label{barrier1}
	The functions 
	\begin{align*}
	w_1(r)&=\frac{n^2+n+2}{8}r,\: r\geq 0,
	\\
	w_2(r)&=\frac{n^2+5n+2}{12}r, r\in\lt[0,\frac{12}{n^2+5n+2}\rt],
	\end{align*}
	are sub-solution and super-solution to Equation \eqref{w=Gn} for $r\geq 0$, respectively.
	\end{proposition}
\begin{proof}
 We compute 
\begin{align*}
\dot{w}_1-G_n(r,w_1)&=\dfrac{n^2+n+2}{8}\lt(1-(1+w_1^2)\dfrac{n-\frac{n^2+n+2}{8}}{\frac{n^2+n+2}{8}-\frac{(n-1)(n-2)}{4}}\rt)
\\
&=-\dfrac{n^2+n+2}{8}w_1^2\leq 0,
\end{align*}
for all $r\geq0$. On the other hand, 
\begin{align*}
\dot{w}_2-G_n(r,w_2)&=\dfrac{n^2+5n+2}{12}\lt(1-(1+w_2^2)\dfrac{n-\frac{n^2+5n+2}{12}}{\frac{n^2+5n+2}{12}-\frac{(n-1)(n-2)}{4}}\rt)
\\
&=\dfrac{n^2+5n+2}{24}(1-w_2^2)\geq 0,
\end{align*}
on $r\in\lt[0,\frac{12}{n^2+5n+2}\rt]$. 
\end{proof}

\begin{remark}
	By the above barriers, it follows that any solution to Equation \eqref{w=Gn} satisfies $w_1\leq w\leq w_2$ for $r\in\lt[0,\frac{12}{n^2+5n+2}\rt]$ and $n\in\set{3,\ldots,6}$. In particular, the term in $G_n(r,w)$ satisfies, 
	\begin{align*}
	\frac12\leq \dfrac{n-\frac{w}{r}}{\frac{w}{r}-\frac{n^2-3n+2}{4}}\leq 1
	\end{align*}
    Consequently, by this estimate, any solution to $\ddot{u}=G_n(r,\dot{u})$ will be convex over $r\in\lt[0,\frac{12}{n^2+5n+2}\rt]$ and $n\in\set{3,\ldots,6}$. 
\end{remark}	
	
\begin{proposition}\label{barrier2}
	The function $w_3(r)=\frac{(n^2+n+2)r}{8\sqrt{1-\lt(\frac{n^2+n+2}{8}r\rt)^2}}$ satisfies
	\begin{align*}
	\dot{w}_3-G_n(r,w_3)\geq 0,\mbox{ for } r\in\lt[0,\frac{8}{n^2+n+2}\rt),
	\end{align*}
\end{proposition}	
\begin{proof}
	It is not hard to check that $w_3$ satisfy $\dot{w}_3=\frac{w_3}{r}(1+w_3^2)$ for $\lt[0,\frac{8}{n^2+n+2}\rt)$. Then, we compute,
	\begin{align*}
	\dot{w_3}-G(r,w_3)=\frac{w_3}{r}(1+w_3^2)\lt(\dfrac{\frac{n^2+n+2}{4\sqrt{1-\lt(\frac{(n^2+n+2)r}{8}\rt)^2}}-\frac{n^2+n+2}{4}}{\frac{n^2+n+2}{8\sqrt{1-\lt(\frac{(n^2+n+2)r}{8}\rt)^2}}-\frac{n^2-3n+2}{4}} \rt).
	\end{align*}
	We note  that  the term in parenthesis is always non-negative in $\lt[0,\frac{8}{n^2+n+2}\rt)$. 	
\end{proof}

\begin{remark}\label{barrier3}
	To ensure existence for a solution we will need a barrier solution to control the term $\partial_y G_n(x,y)$.  First, we calculate 
	\begin{align*}
	&\partial_yG_n(x,y)
	\\
	=&\dfrac{1}{x\lt(y-\frac{n^2-3n+2}{4}x\rt)^2}\lt(-y^2+y^3(2nx-1)-2y^4-\frac{n(n^2-3n+2)}{4}x^2+\frac{n^2-3n+2}{2}xy\rt.
	\\
	&\lt.+\frac{3(n^2-3n+2)}{4}y^3x-\frac{3n(n^2-3n+2)}{4}x^2y^2+\frac{n^2-3n+2}{4}xy^2\rt).
	\end{align*}
	We note that for the barriers found in Proposition \ref{barrier1}, the denominator in $\partial_yG_n(x,y)$ is $O(x^3)$. On the other hand, the numerator has terms of order $O(x^2)$. Namely, 
	\begin{align*}
	L(x,y)=-y^2-\dfrac{n(n^2-3n+2)}{4}x^2+\dfrac{n^2-3n+2}{2}xy.
	\end{align*}
	Therefore, by choosing the function 
	\begin{align*}
	w_4(r)=\dfrac{\sqrt{n^4-4n^3+7n^2-8n+4}}{2\sqrt{6}}r,
	\end{align*}
   we can estimate the above line by 
	\begin{align*}
	L(r,w)\leq-w_4^2-\dfrac{n(n^2-3n+2)}{4}r^2+\dfrac{n^2-3n+2}{2}rw_2=0.
	\end{align*}
	Moreover,  it is not hard to check that the function $w_4$ is a sub solution to Equation \eqref{w=Gn} for $n\geq 3$.  In particular, 
	\begin{align}
	\partial_yG_n(r,w)\leq C(n)r,
	\end{align}
	for $r\in \lt[0, \dfrac{12}{n^2+5n+2}\rt]$ and $w_4\leq w\leq w_3$. 
\end{remark}

\begin{proposition}\label{Prop w exist}
Let $n\in\set{3,\ldots,6}$. There exists $R(n)>0$ and a unique continuous solution $w:[0, R)\to\rr$ to equation $\dot{w}=G_n(r,w)$  with $w(0)=0$. 
\end{proposition}

\begin{proof}
	We define the closed space 
	\begin{align*}
	X=\set{w\in\mathcal{C}\lt(\lt[0,R(n)\rt]\rt): w(0)=0,\:w_4(r)\leq w(r)\leq w_3(r)},
	\end{align*}
	where $R(n)<R_1(n)=\frac{12}{n^2+5n+2}$ is still to be fixed. Here we are considering all barriers found in propositions \ref{barrier1}, \ref{barrier2} and Remark \eqref{barrier3}. Now we define the operator 
	\begin{align*}
	T:X&\to\mathcal{C}([0,R(n)])
	\\
	w(r)&\to T(w)(r)=\int\limits_0^r\frac{w}{s}(1+w^2)\frac{n-\frac{w}{s}}{\frac{w}{s}-\frac{(n-1)(n-2)}{4}}ds.
	\end{align*}
	It is not hard to see that $T(X)\subset X$. Indeed, the terms $\dfrac{n-\frac{w}{s}}{\frac{w}{s}-\frac{n^2-3n+2}{4}}$  and $\frac{w}{s}(1+w^2)$ are decreasing and increasing in $w$, respectively. Then, it follows
	\begin{align*}
	T(w)-w_4&=\int\limits_0^r\frac{w}{s}(1+w^2)\dfrac{n-\frac{w}{s}}{\frac{w}{s}-\frac{(n-1)(n-2)}{4}}-\dot{w}_4\:ds
	\\
	&\geq\int\limits_0^r\frac{w}{s}(1+w^2)\frac{n-\frac{w_2}{s}}{\frac{w_2}{s}-\frac{(n-1)(n-2)}{4}}-\dot{w}_4ds
	\\
	&=\int\limits_0^r\frac{w}{2s}(1+w^2)-\dot{w}_4ds
	\\
	&\geq\int\limits_0^r\frac{w_4}{2s}(1+w_4^2)-\dot{w}_4ds
	\\
	&=\int\limits_0^r\frac{w_4^3}{2s}ds>0. 
	\end{align*}
On the other hand, we have
	\begin{align*}
	T(w)-w_3&=\int\limits_0^r\frac{w}{s}(1+w^2)\dfrac{n-\frac{w}{s}}{\frac{w}{s}-\frac{(n-1)(n-2)}{4}}-\dot{w}_3\:ds
	\\
	&\leq\int\limits_0^r\frac{w}{s}(1+w^2)\dfrac{n-\frac{w_1}{s}}{\frac{w_1}{s}-\frac{(n-1)(n-2)}{4}}-\dot{w}_3ds
	\\
	&=\int\limits_0^r\frac{w}{s}(1+w^2)-\dot{w}_3ds
	\\
	&\leq\int\limits_0^r\frac{w_3}{s}(1+w_3^2)-\dot{w_3}ds\leq 0.
	\end{align*}
	Now we choose $R>0$ such that $T$ is a contraction on $X$. Indeed, let $w,v\in X$ and we compute 
	\begin{align*}
	\abs{T(w)(r)-T(v)(r)}\leq \norm{w-v}_{\infty}\abs{\int\limits_0^{r}\partial_yG_nds}\leq C(n)\norm{v-w}_{\infty}\int_0^{r}sds.
	\end{align*}
	Therefore, since the last integral is finite by the Remark \ref{barrier3}, we set $R_2(n)$ such that 
	\begin{align*}
	C(n)\int_0^{R_2(n)}sds<1.
	\end{align*}
	Finally, by choosing $R(n)=\min\set{R_1(n),R_2(n)}$, it follows that $T:X\to X$ is a contraction. Then, by Banach's Fixed Point Theorem, we obtain the existence and uniqueness of the solution to Equation \eqref{w=Gn} for $n\in\set{3,\ldots,6}$ and $r\in[0,R(n)]$.   
\end{proof}

\begin{remark}
	Since the term in equation \eqref{w=Gn} satisfies
	\begin{align*}
	\dfrac{n-\frac{w}{r}}{\frac{w}{r}-\frac{n^2-3n+2}{4}}\geq \frac{1}{2},
	\end{align*}
it follows that the solution $w$ found in Proposition \ref{Prop w exist} is bounded for below by  the function
\begin{align*}
w_5=\dfrac{a\sqrt{r}}{\sqrt{1-a^2r}}, \mbox{ for any }a>0.
\end{align*}	
Note that $w_5$ satisies the equation
\begin{align*}
w_5=\frac{w_5}{2r}(1+w_5^2).
\end{align*}
In particular, by choosing $a=\sqrt{\frac{n^2+n+2}{8}}$, it follows that the solution $w$ will be convex and will satisfy $\lim\limits_{r\to\frac{8}{n^2+n+2}}w(r)=\infty$.   
\end{remark}

\begin{figure} \label{fig}
	\centering
	\includegraphics[width=0.5\columnwidth]{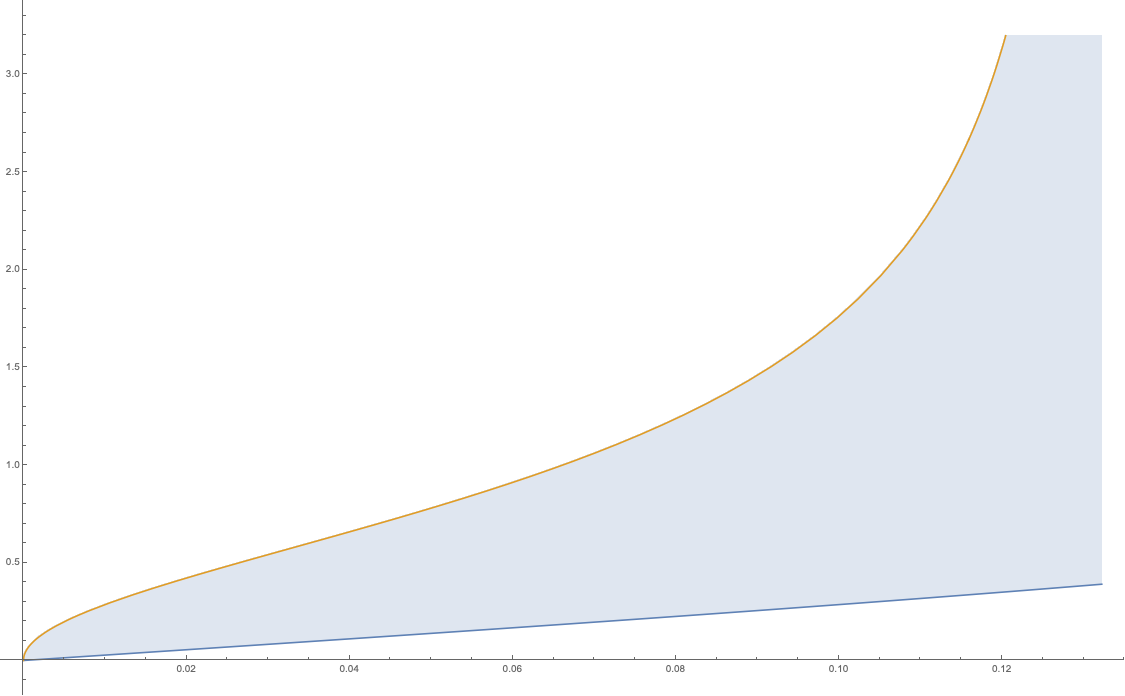}
	\caption{ The barriers $w_3$ and $w_5$.}
\end{figure}

\end{document}